\documentclass[]{article}

\usepackage{epstopdf}
\usepackage[caption=false]{subfig}

\usepackage{mathtools,hyperref,amsmath,amsthm,amssymb}

\usepackage[numbers,sort&compress]{natbib}
\bibpunct[, ]{[}{]}{,}{n}{,}{,}
\renewcommand\bibfont{\fontsize{10}{12}\selectfont}
\makeatletter
\def\NAT@def@citea{\def\@citea{\NAT@separator}}
\makeatother

\theoremstyle{plain}
\newtheorem{theorem}{Theorem}[section]
\newtheorem{question}[theorem]{Question}
\newtheorem{lemma}[theorem]{Lemma}
\newtheorem{corollary}[theorem]{Corollary}
\newtheorem{proposition}[theorem]{Proposition}

\theoremstyle{definition}
\newtheorem{definition}[theorem]{Definition}

\theoremstyle{remark}
\newtheorem{remark}{Remark}

\newcommand{\R}{\mathbb{R}}

\newcommand{\del}{\partial}

\renewcommand{\gg}{\gamma}

\newcommand{\gd}{\delta}
\newcommand{\gs}{\sigma}

\newcommand{\gl}{\lambda}

\newcommand{\ga}{\alpha}
\newcommand{\gb}{\beta}

\newcommand{\ep}{\varepsilon}

\renewcommand{\bar}[1]{\overline{#1}}

\newcommand{\hatu}{\hat{u}}
\newcommand{\hatuy}{\hat{u}_{y_0}}

\DeclareMathOperator{\tr}{tr}

\DeclareMathOperator{\diam}{diam}

\DeclareMathOperator{\dist}{dist}

\begin{document}


\title{A Pogorelov-type estimate for the real twisted Monge-Amp\`ere equation}

\author{Joshua Jordan \thanks{CONTACT J. Jordan. Email: joshua.jordan@vanderbilt.edu}}

\maketitle

\begin{abstract}
We establish a Pogorelov-type interior $C^2$ estimate for the real twisted Monge–Amp\`ere equation without assuming uniform ellipticity. This equation is fully nonlinear and elliptic, but is neither concave nor a function of the eigenvalues of the Hessian, so the standard techniques for interior second-derivative estimates do not directly apply. Building on the local partial Legendre transform approach of Streets and Warren, we derive a differential identity for the associated positive-definite matrix $W$ and use some linear algebra to control the third-order terms in a maximum principle argument. As an application, we prove a Liouville theorem for entire split-convex solutions with quadratic growth, extending the rigidity theorem of Streets and Warren by replacing their uniform ellipticity hypothesis with suitable quadratic growth assumptions.
\end{abstract}



\tableofcontents

\section{Introduction}
        The real twisted Monge-Amp\`ere (tMA) equation is a non-concave fully nonlinear partial differential equation which was defined by Streets and Warren in \cite{streets_evans-krylov_2016}. Its complex counter-part arises in the study of nonlinear sigma models \cite{hull_generalized_2010} and non-K\"ahler Ricci-flat metrics \cite{streets_pluriclosed_2015-1}. In this paper, we will prove the first Pogorelov-style interior $C^2$ estimates for this equation. This estimate will lead naturally to a Liouville theorem for quadratic growth solutions, generalizing the Liouville theorem for uniformly elliptic solutions \cite[][Corollary 2.4]{streets_evans-krylov_2016}.
        
        To define the equation and the associated family of admissible functions, we require a splitting $\R^n = \R^{n_+}\times \R^{n_-}$ where $(x^i)_{i=1}^{n_+}$ will be coordinates on the $\R^{n_+}$ factor and $(y^\ga)_{\ga = 1}^{n_-}$ on the $\R^{n_-}$ factor. For a connected open set $\Omega \subset \R^n$, the \textit{real tMA equation} is
	    \begin{equation}\label{eq:TMA}
	        \det(D_x^2u) = e^{f(x,y)}\det (-D_y^2u)\text{ in }\Omega
	    \end{equation}
        where $f\in C^\infty(\Omega)$.

        We will denote the linearized operator at the solution $u$ by $L$. It is given by
	    \begin{equation} \label{eq:linTMA}
		      L\phi \coloneq u^{ij}\phi_{ij} - u^{\ga\gb}\phi_{\ga\gb}.
	    \end{equation}
	    Therefore, (\ref{eq:TMA}) is strictly elliptic at a twice-differentiable function $u$ whenever it is \textit{split-convex}, i.e. $D_x^2u > 0$ and $D_y^2 u < 0$. As such, we define the \textit{set of admissible functions} to be
        \begin{equation}\label{eq:split-cvx}
            \mathcal{A}(\Omega) \coloneq \{u\in C^{\infty}(\Omega)| D_x^2u>0,\, D_y^2u<0\}.
        \end{equation}
        
        Whenever $\min\{n_+,n_-\} = 0$, this equation reduces to the usual Monge-Amp\`ere equation and (\ref{eq:split-cvx}) reduces to the set of either convex or concave functions. In other words, in the case of a degenerate factor in the splitting, the existence, regularity, and rigidity questions reduce to well-known fully nonlinear PDE theory \cite{Jorgens1954,calabi_improper_1958,pogorelov_improper_1972,caffarelli_dirichlet_1985-1,caffarelli-li,cheng-yau,evans_classical_1982,krylov_boundedly_1984}.

        The study of the real tMA began with Streets-Warren's paper \cite{streets_evans-krylov_2016} where they discovered a convexity principle strong enough to prove Evans-Krylov regularity in the presence of uniform ellipticity. Implicit in their work are natural analogues to the quantities used for the Pogorelov estimate \cite{pogorelov_improper_1972}. Namely, by way of the partial Legendre transform, they find that (\ref{eq:TMA}) can be locally transformed into a Monge-Amp\`ere equation for the PLT $\hatu$, and they find a matrix $W(D^2u)$ and a vector $(D_xu, -y)$ which correspond to the Hessian and gradient of $\hatu$. A result of Mooney and Savin on interior $C^2$ estimates for the PDE $(u_{11}u_{22})^{\frac{1}{2}} = 1$ similarly uses a test function motivated by the PLT \cite{mooney_regularity_2019}, but they are able to work purely with a modification of the gradient terms.
        
        Equation (\ref{eq:TMA}) sits in an interesting gap in the literature on fully nonlinear elliptic PDEs. In particular, the real tMA is neither Hessian nor concave. To be more explicit: 
        \begin{enumerate}
            \item Equation \ref{eq:TMA} is not \textit{Hessian}, i.e. not an equation of the eigenvalues of the Hessian matrix, and
            \item Equation \ref{eq:TMA} is not a \textit{concave} function of the Hessian matrix.
        \end{enumerate}
        This leads to two critical issues in the study of this equation. First, boundary estimates are very hard to obtain. And second, the standard interior $C^2$-estimate techniques and Evans-Krylov theory are not applicable.

        The reduced symmetry group that comes from failing to be a Hessian PDE makes obtaining boundary estimates very difficult. To see why, consider the seminal work of Caffarelli-Nirenberg-Spruck \cite{caffarelli_dirichlet_1985-1}. To obtain estimates on the tangent-normal part of the Hessian on the boundary of a strictly convex domain, the authors make crucial use of the fact that the equation is invariant under $O(n)$ and the orbits of non-zero vectors are spheres $\mathbb{S}^{n-1}$. By contrast, the real tMA equation's symmetry group is only $O(n_+)\times O(n_-)$ whose orbits are products of spheres $\mathbb{S}^{n_+ - 1}\times \mathbb{S}^{n_-1}$. Since the sphere product is codimension 2, it cannot be used to obtain estimates at points which e.g. have a mixed-type normal. The paper of Mooney-Savin \cite{mooney_regularity_2019} overcomes a related deficiency by way of an approximation scheme \`a la Lions \cite{lions}. In any case, an interior $C^2$-estimate is a crucial first step in any such program.
        
        To the author's knowledge, this is the first interior $C^2$-estimate for a non-uniformly elliptic, fully nonlinear, elliptic partial differential equation which is neither a function of the eigenvalues of the Hessian nor a concave function of the Hessian matrix in dimension larger than 2. To substantiate this claim, it is useful to review some nearby papers and place them in context.
        
        First, non-concave, \textit{uniformly elliptic} operators had seen systematic study in the 2000's by the Caffarelli school. Speaking very roughly, these papers focus on estimates in $C^{2,\ga}$. These estimates represent generalizations of the celebrated Evans-Krylov theorem \cite{evans_classical_1982,krylov_boundedly_1984} and have been incredibly important to the study of the special Lagrangian equation and the class of Isaacs-Bellman equations. Caffarelli and Yuan showed interior $C^{2,\ga}$-estimates for uniformly elliptic operators with convex level sets \cite{caffarelli_priori_2000}. Yuan was able to use geometric arguments to show interior $C^{2,\ga}$-estimates for the non-uniformly elliptic, non-convex special Lagrangian equation, but this equation is Hessian \cite{yuan}. Cabr\'e and Caffarelli obtained $C^{2,\ga}$-estimates for a minimum of a convex and a concave uniformly elliptic operator \cite{cabre_interior_2003}. Work on these types of estimates slowed somewhat after the counter-examples of Nadirashvili and Vl\u{a}du\c{t}, which demonstrated that even smooth, uniformly elliptic, non-concave equations admit non-classical solutions, thus obstructing some combination of the smooth Dirichlet problem and interior regularity \cite{nadirashvili_singular_2011}.
        
        Motivated by questions in complex differential geometry \cite{streets_pluriclosed_2015-1}, Streets and Warren proved an interior $C^{2,\ga}$-estimate for the real and complex tMAs under the assumption that the solution was uniformly elliptic. Collins, extending the methods of Cabr\'e-Caffarelli, then proved interior $C^{2,\ga}$-estimates for \textit{twisted-type equations} \cite{collins_c2alpha_2016}, of which real tMA is an example, under a uniform ellipticity hypothesis. Critically, we do not take uniform ellipticity as a hypothesis.

        Next, there is an ample body of literature on non-uniformly elliptic \textit{concave Hessian equations}. The smooth Dirichlet problem for equations of this type on convex domains was resolved in the foundational series of papers by Caffarelli-Nirenberg-Spruck and Caffarelli-Kohn-Nirenberg-Spruck \cite{caffarelli_dirichlet_1985-1,caffarelli_dirichlet_1985-2,caffarelli_dirichlet-3}. Their method rests on establishing global / boundary estimates using the invariance of the equation. Interior estimates for equations of this type go back to the work of J\"orgens ($n=2$) \cite{Jorgens1954}, Calabi ($n\leq 6$) \cite{calabi_improper_1958}, and Pogorelov ($1\leq n$) \cite{pogorelov_improper_1972} on the Monge-Amp\`ere equation. Cheng and Yau \cite{cheng-yau} refined this work. Similar results are established by Chou and Wang for the $\sigma_k$ equations \cite{chou_variational_2001} (see also \cite{warren-yuan} for a more geometric approach or \cite{guan-qiu_prescribing-scalar} with semi-convexity). Shankar-Yuan and Shankar have been able to obtain $C^2$ estimates without relying directly on concavity, but by using a doubling inequality and compactness \cite{shankar_regularity_2023,shankar-Yuan_Hessian,shankar-c2}. In these cases, concavity is often used implicitly to obtain a useful sign on bad third order terms. As mentioned before, we do not have concavity or Hessianness available to us. 
        
        Finally, there is much work on interior $C^2$-estimates for equations which are concave, but not Hessian. Mooney and Savin have studied the concave, but not Hessian, equation $(u_{11}u_{22})^{\frac{1}{2}}=1$ and have obtained interior $C^2$-estimates which fail for the higher dimensional version of the equation \cite{mooney_regularity_2019}. Chen has obtained interior $C^2$-estimates for a large class of concave equations involving lower order terms \cite{chen_local_2005}. 

        Our proof will rely crucially on the fact, established in \cite{streets_evans-krylov_2016}, that (\ref{eq:TMA}) is equivalent to $\det W(D^2u) = e^f$ for the matrix 
        $$W(D^2u) = \begin{pmatrix}
		D_x^2u - (D_{xy}^2u)^T(D_y^2u)^{-1}D_{xy}^2u & (D_{xy}^2u)^T(D_y^2u)^{-1}\\
		(D_y^2u)^{-1}D_{xy}^2u & -(D_y^2u)^{-1}
	    \end{pmatrix}.$$
        When the partial gradient map is globally invertible, this matrix is just the pullback of the Hessian of the partial Legendre transform of $u$, i.e. $D^2_{(x,p)}\hatu(x,p)$. 
        
        It is worthwhile to remark that it seems plausible that one could find a proof of our estimate using Pogorelov using the dualized coordinates whenever the partial gradient map is globally invertible. However, our proof avoids any hypothesis of the global invertibility of this map. This will hopefully make it possible for our results to generalize to lower regularity scenarios where the gradient map is not globally defined or to other non-concave equations with ``proper concavity principles" like $LW \geq 0$ which are not induced by a diffeomorphism (e.g. the complex tMA equation).
        
        To obtain the interior estimate, we will expand Streets-Warren's subsolution property $LW\geq 0$ \cite{streets_evans-krylov_2016} to find the explicit form of the right-hand side in terms of $g$ and $W$. We will then apply a maximum principle argument to a modification of Pogorelov's test function, however $W$ and $g$ are not simultaneously diagonalizable so the calculations are substantially more involved. In addition, as the equation is not $O(n)$ invariant, we will not be able to rotate the eigenvector of the maximal eigenvalue to point in a coordinate direction. This argument will yield the main theorem.
	    \begin{theorem}\label{thm:main}
            Consider $0\in \Omega \subset \R^{n}$, a bounded open set in $\R^n$. Suppose that $u \in \mathcal{A}(\Omega)$ is a solution of (\ref{eq:TMA}) normalized so that $u(0) = |Du(0)| = 0$ and $\hatu(x,y) = u(x,y) - y^T D_yu(x,y)$. If $S_r(0;\Omega)\Subset \Omega$, then the estimate
            \begin{equation*}
                \sup_{S_r(0)}\left(\left(\frac{\rho}{r^2}\right)^{\gs}|W(D^2u)|\right) \leq C(n,r,\|D_xu\|_{C^0(S_r)}, \diam_y S_r, \|f\|_{C^2(S_r)})
            \end{equation*}
            holds with weight function $\rho(x,y) \coloneq r^2 - \hatu(x,y)$. The exponent $\gs$ is given by
            \begin{equation*}
                \gs = 2 + 4\|D_y^2f\|_{C^0(S_r(0))}M_r,\quad M_r = \sup_{S_r}(|D_xu|^2 + |y|^2)
            \end{equation*}
            and both $\gs$ and the constant $C$ are invariant under elliptic rescalings $u(x,y) \mapsto \mu^{-2}u(\mu x, \mu y)$. See (\ref{def:Wmatrix}) for the explicit dependence of $W$ on $D^2u$.
	    \end{theorem}
        
        We state the $f\equiv const.$ version as a corollary. It will end up being crucial for proving the Liouville theorem.
        \begin{corollary}\label{cor:fconst}
            When $f\equiv const.$, we have $\gs = 2$ and the estimate becomes 
            \begin{equation*}
                \sup_{S_r(0)}\left(\left(\frac{\rho}{r^2}\right)^{2}|W(D^2u)|\right) \leq C(n)\frac{M_r}{r^2}.
            \end{equation*}
        \end{corollary}
        When $n_- \equiv 0$ and $f = f(x)$, one sees that we recover the Pogorelov estimate for convex solutions of the Monge-Amp\`ere equation with a weaker exponent \cite[see][]{chou_variational_2001,figalli-2014}. 
        \begin{equation*}
            \sup_{S_r}((r^2-u)^{2}|D^2u|) \leq C(n)\|Du\|_{C^0(S_r)}^2r^{2}.
        \end{equation*}
        Our result can, therefore, be viewed as generalizing the Pogorelov estimate to tMAs with $n_- > 0$. 

        As an application of Corollary \ref{cor:fconst}, we strengthen the quadratic rigidity theorem of Streets-Warren \cite[][Theorem 1.1]{streets_evans-krylov_2016} by relaxing the hypothesis of global uniform ellipticity to a hypothesis of quadratic growth.
	    \begin{corollary}\label{cor:conditional-JCP}
		      Suppose that $u \in \mathcal{A}(\R^n)$ is a solution of (\ref{eq:TMA}) with $e^f\equiv c_0$, a positive constant, normalized so that $u(0) = |Du(0)| = 0$.  If there are constants $A,C>1$ such that both 
              \begin{align*}
            &\ |u|(x,y) \leq A(|x|^2+|y|^2),\\
            C^{-2}(|x|^2 + |y|^2)\leq&\ u(x,y) - y^TD_yu(x,y) \leq C^2(|x|^2 + |y|^2),
            \end{align*}
            then $u$ is a quadratic polynomial. 
	    \end{corollary}
        The proof is relatively standard. The split-convexity implies a gradient bound which, when taken with the hypotheses, will imply that $M_r = O(r^2)$. Then an application of Corollary \ref{cor:fconst} on nested sections and the fact that $\rho^2 = O(R^4)$ will imply the uniform bound, which can be plugged directly into the Streets-Warren theorem.
        
        The hypotheses of Corollary \ref{cor:conditional-JCP} reduce  simply to the assumption of quadratic growth of $u$ whenever $n_- = 0$ (the Monge-Amp\`ere case). The proof of this growth rate is central to, for example, Caffarelli-Li's extension of the J\"orgens-Calabi-Pogorelov theorem \cite{caffarelli-li}. Similar sorts of growth hypotheses are common in the Liouville theorem literature \cite[see e.g.][]{mooney_regularity_2019,warren-yuan2,ecker_bernstein_1990}. 
        
        We also note that the bound on $|u|$ does not imply the growth hypothesis on $u-y^TD_yu$. Even in the $n_+ = n_- = 1$ case,
        $$u(x,y) = \sqrt{1+x^2} - \sqrt{1+y^2}$$
        is split-convex and quadratically trapped without having $u-y^TD_yu$ of quadratic growth.
        
        However, many of the usual sorts of test functions are either inadmissible or rigid, the author is tempted to pose the question:
        \begin{question}
            Do there exist classical non-quadratic entire solutions of (\ref{eq:TMA})?
        \end{question}

\section{Background \& Notation}
	
    \subsection{Partial Legendre Transform Review}\label{subsec:PLT}
    The intent of this subsection is to introduce the partial gradient map and motivate the $W$ matrix and several other quantities that will show up in test functions later. It will be seen that all of these quantities have natural interpretations in the image coordinates $(x,p)$.

    In what follows, $\Omega \subset \R^{n_+}\times \R^{n_-}$ is an open domain and $\pi_{\pm}:\Omega \to \R^{n_\pm}$ is the linear projection map. Given a function $u \in C^1(\Omega)$, the \textit{partial gradient map} is the mapping
	\begin{align}\label{eq:PLTdiffeo}
		\Phi_u: \Omega & \to \R^{n_+}_x \times \R^{n_-}_p\\
		(x,y)& \mapsto (x,p(x,y)) = (x, D_yu(x,y)) \nonumber.
	\end{align}
    
    The Jacobian matrix of $\Phi$ is
	$$\Phi_* = \begin{pmatrix}
		I & 0\\
		D_{xy}^2u & D_y^2 u
	\end{pmatrix}.$$
    Whenever $u\in \mathcal{A}(\Omega)$, $\Phi_*$ is invertible at each point. We name the inverse matrix
	$$\Psi_* \coloneq (\Phi_*)^{-1} = \begin{pmatrix}
		I & 0\\
		-(D_y^2 u)^{-1}D_{xy}^2u & (D_y^2u)^{-1}
	\end{pmatrix}.$$

    Thus, we can use the inverse function to define a local \textit{partial Legendre transform based at $y_0$} of $u\in\mathcal{A}(\Omega)$ as the function
        \begin{equation*}
        \hatuy(x,p) \coloneq u(x,y(x,p)) - (y(x,p) - y_0)^Tp.
    \end{equation*}
    Here the inverse function $y(x,p)$ is well-defined only near $(x_0,p_0) \in \Phi(\Omega)$. Interestingly, even though the potential is only locally defined, one can compute its first and second derivatives and they are defined globally.
    
    We may use the chain rule to compute the $(x,p)$-differential and the Hessian of the local transform (see \cite{streets_evans-krylov_2016} for details).
    \begin{align}
        d\hatuy(x,p) =&\ \begin{pmatrix}
		D_xu & (y_0 - y) \end{pmatrix} \label{eq:xp-differential}\\
        (D^2_{(x,p)}\hatuy)(x,p) =&\ \begin{pmatrix}
		D_x^2u - (D_{xy}^2u)^T(D_y^2u)^{-1}D_{xy}^2u & (D_{xy}^2u)^T(D_y^2u)^{-1}\\
		(D_y^2u)^{-1}D_{xy}^2u & -(D_y^2u)^{-1}
	   \end{pmatrix}
    \end{align}
    Notice that the Hessian of $\hatuy$ at $p = p(x,y)$ is a coordinate change away from a positive matrix of second derivatives of $u \in \mathcal{A}(\Omega)$, i.e.
	\begin{equation}\label{eq:pos}
		(\Psi_*)^T\begin{pmatrix}
		D_x^2u(x,y) & 0\\
		0 & -D_y^2 u(x,y)
	\end{pmatrix}\Psi_* = D_{(x,p)}^2\hatu(x,p).
	\end{equation}
	Since these matrices transform as quadratic forms, we will treat them as Riemannian metrics, i.e.
    \begin{equation}\label{def:Wmatrix}
	   g \coloneq \begin{pmatrix}
	       	D_x^2u & 0\\
		      0 & -D_y^2u
	   \end{pmatrix}, \quad W \coloneq \begin{pmatrix}
	       D_x^2u - (D_{xy}^2u)^T(D_y^2u)^{-1}D_{xy}^2u & (D_{xy}^2u)^T(D_y^2u)^{-1}\\
	       (D_y^2u)^{-1}D_{xy}^2u & -(D_y^2u)^{-1}
	\end{pmatrix}.
    \end{equation}
	
	Moreover, (\ref{eq:pos}) can be rewritten in terms of pullbacks as $\Psi^*g = W$, telling us that $g > 0$ (i.e. $u$ is strictly split-convex in $(x,y)$-space) iff $W>0$ (i.e. $\hatu$ is strictly convex in $(x,p)$-space) and also that
	\begin{equation}\label{eq:PLTform}
		\det W = (\det \Psi_*)^2 \det g = \frac{\det D_x^2u}{\det(-D_y^2 u)}.
	\end{equation}
	Thus, whenever $u\in \mathcal{A}(\Omega)$ has a globally defined inverse $\Psi$ to the partial gradient map $\Phi$, then $u$ solves (\ref{eq:TMA}) in $\Omega$ iff $\hatu(x,p)$ solves 
    $$\det D^2\hatu(x,p) = e^{f\circ \Psi} > 0.$$
    
    This correspondence to the Monge-Amp\`ere equation is central to Streets and Warren's Evans-Krylov theorem for the twisted Monge-Amp\`ere equations \cite{streets_evans-krylov_2016}. They are able to leverage the concavity of the log-determinant to get that $LW\geq 0$, a subsolution condition on $W$, from which they run the usual Evans-Krylov-style argument.

    \subsection{Notation}
    Repeated indices given by capital Latin letters, e.g. $A,B,C$, are to be summed over the index set $\{1,...,n\}$, i.e. one sums over the whole $\R^n$. Lower case indices are summed only over part of the splitting. For latin letters, e.g. $i,j,k$, the sum should be taken over the index set $\{1,...,n_+\}$, i.e. only over a basis for $\R^{n_+}$. Conversely, lowercase greek letters should be summed over $\{1,...,n_-\}$.

    In this notation, the exterior derivative is written 
    $$df = f_i dx^i + f_\ga dy^\ga.$$
    We will use fairly standard differential geometry notation to denote components and derivatives, e.g.
    $$S_{i\ga k, \gb} = \frac{\del}{\del y^\gb} S(\frac{\del}{\del x^i}, \frac{\del}{\del y^\ga},\frac{\del}{\del x^k})$$
    so the indices correspond to evaluation and commas denote partial derivatives.

    I will occasionally, primarily in \S 4, make use of the mild abuse of notation 
    $$d\hatuy(x,p) = \begin{pmatrix}
		D_xu & (y_0 - y) \end{pmatrix}$$
    though a global choice of $\hatu(x,p)$ may not necessarily exist. The covector itself is defined independently of such a choice and this choice of notation is much more economical than repeated writing the covector out.

    In addition, $W$ and $g$ will always refer to the matrices of second derivatives of $u$ given by
    \[
	   g = \begin{pmatrix}
	       	D_x^2u & 0\\
		      0 & -D_y^2u
	   \end{pmatrix}, \quad W = \begin{pmatrix}
	       D_x^2u - (D_{xy}^2u)^T(D_y^2u)^{-1}D_{xy}^2u & (D_{xy}^2u)^T(D_y^2u)^{-1}\\
	       (D_y^2u)^{-1}D_{xy}^2u & -(D_y^2u)^{-1}
	\end{pmatrix}.
   \]
   In these formulas, $D_{xy}^2u$ is to be viewed as an $n_- \times n_+$ matrix, i.e. as a linear transformation $\R^{n_+}\to \R^{n_-}$.
   
    Since $g$ is identified with the ``pure-type" Hessian of $u$, we note that $g>0$ implies that $D_x^2u$ and $D_y^2u$ are both invertible. Due to this, we will reserve superscripts on $u$ to refer to elements of the inverse pure-type Hessians. To be specific, 
    $$g = \begin{pmatrix}
        u_{ij} & 0 \\
        0 & -u_{\ga\gb}
    \end{pmatrix}, \quad g^{-1} = \begin{pmatrix}
        u^{ij} & 0 \\
        0 & -u^{\ga\gb}
    \end{pmatrix}.$$
    
    We will use a slightly different convention for superscript indices on $W$ and $g$. These will refer to components of $W^{-1}$ and $g^{-1}$, i.e.
    $$W=(W_{AB}),\, g=(g_{AB}) \quad W^{-1} = (W^{AB}),\, g^{-1}= (g^{AB}).$$
    Both of these conventions show up in the coordinate expression of $W$, but they are standard in the literature (see either \cite{streets_evans-krylov_2016,garcia-fernandez_generalized_2021})
   
    $X\Subset Y$ indicates $X \subset Y$ and $\bar{X}$ is compact in $Y$.

    As we are aiming to prove the estimate purely using calculations in $\Omega$, we will introduce a small abuse of notation. Namely,
    \begin{equation*}
        \hatuy(x,y) \coloneq \hatuy(x,p(x,y)) = u(x,y) - (y-y_0)^TD_yu(x,y)
    \end{equation*}
    and we will drop the coordinates when they can be inferred from context. This sort of abuse of notation is standard in differential geometry.

    We will use the geometer's notation for Jacobians, namely $\Phi_*$ will denote the action of the Jacobian on a vector, and $\Phi^* = (\Phi_*)^T$ is the adjoint action on a covector.
    
	\section{Sections}
    We will define a section in the vein of Caffarelli.
    \begin{definition}
        If $u \in \mathcal{A}(\Omega)$ and $(x_0,y_0)\in \Omega$, then we define \textit{the section of $u$ of radius $r$ centered at $(x_0,y_0)$}, which will be denoted by $S_r((x_0,y_0);\Omega)$, to be the connected component of the set 
        $$\{(x,y) \in \Omega | \hatuy(x,y) < \hatuy(x_0,y_0) + (x-x_0)^TD_xu(x_0,y_0) + r^2\}$$
        containing the point $(x_0,y_0)$. We will often drop $\Omega$ and $(x_0,y_0)$ when they can be inferred.
    \end{definition}

    It is useful to note that, although $\hatu$ is generally not a convex function, as long as its sections are precompact and normalized, it has an isolated minimum at the center.
    \begin{lemma}\label{lem:hatu-pos}
    Suppose $u\in\mathcal{A}(\Omega)$. If $S_r(0; \Omega) \Subset \Omega$ and $u(0) = |Du(0)| = 0$, then $\hat{u}(x,y) > 0$ on $S_r(0) \setminus \{0\}$.
        \begin{proof}
            By taking the gradient, we see that critical points satisfy the system
            \[
            \begin{cases}
                0= u_i - y_\ga u_{i\ga},\\
                0= -y_\ga u_{\ga\gb}.
            \end{cases}
            \]
            Notice first, that $(0,0)$ is a critical point.

            Second, since $u\in \mathcal{A}(\Omega)$ we know that $D_y^2u$ is invertible. Thus, every critical point must occur on the linear subspace $y = 0$. Moreover, any such critical point must be a non-degenerate local minimum as the hessian is
            \begin{equation}\label{eq:nondegmin}
            D^2\hatu(x,0) = \begin{pmatrix}
                D_{x}^2u & 0\\
                0 & -D_y^2u
            \end{pmatrix} > 0.
            \end{equation}
            In particular, $(0,0)$ is a strict local minimum. Moreover, every critical point occurs (1) where $y=0$ and (2) is a non-degenerate local minimum.

            Thus, by the Mountain Pass Lemma the result follows \cite[][Chapter 2, Theorem 1.1]{struwe-book}.
            
        \end{proof}
    \end{lemma}
	
	\section{Interior Second Derivative Estimate (Proof of Theorem \ref{thm:main})}
	Given the relationship in (\ref{eq:PLTform}), we seek a Pogorelov-style interior $C^2$ estimate for strictly split-convex solutions of (\ref{eq:TMA}). A natural choice of objective function in this case is then
	$$\tilde{Q}(x,y,w) = \sigma\log\rho(x,y) + \log \phi(\frac{1}{2}(|D_xu|^2 + |y|^2)) + \log w^TW(D^2u)w$$
    where $(x,y) \in S_r(0;\Omega)$, $\rho(x,y) = r^2 - \hatu(x,y)$, and $w = w^+ + w^- \in \mathbb{S}^{n-1} \subset \R^{n_+}\times \R^{n_-}$. Before expanding on the function much more, it is worthwhile to notice that this is precisely the ``dual" of the usual Pogorelov test function \cite{pogorelov_improper_1972}. In $\tilde{Q}$, $W$ is the pull back of the Hessian of the PLT function $D^2_{(x,p)}\hatu(x,p)$, $|D_xu|^2 + |y|^2$ is its gradient norm $|D_(x,p)\hatu(x,p)|^2$, and $\rho$ is the obvious choice. 
    
    Moreover, since $\bar{S}_r(0)\times \mathbb{S}^{n-1}$ is compact and $e^{\tilde{Q}}$ can be continuously extended to $0$ on $\del S_r \times S^{n-1}$, a positive maximum must be obtained at a point $(x',y',v) \in S_r(0)\times \mathbb{S}^{n-1}$ in the interior. Since $e^{\tilde{Q}}$ is positively proportional to the Rayleigh quotient at every point, $v$ must be a unit-length eigenvector of $W(D^2u)(x',y')$ corresponding to $\gl$, the maximal eigenvalue of $W(D^2u)(x',y')$. We will proceed by freezing $w = v$ and considering instead the function $Q(x,y) = \tilde{Q}(x,y,v)$ which is still maximized at $(x',y')$. The author would like to note the influence of several works on Monge-Amp\`ere and Hessian equations \cite{chou_variational_2001, guan-qiu_prescribing-scalar,shankar_regularity_2023,yuan_monotonicity_2023} which motivated the construction of this test function.

    One last note before undertaking the calculation. The third order terms in the Monge-Amp\`ere case are typically dealt with by rotating and diagonalizing so that the eigendirection is a coordinate direction. This will not be possible in our case. Additionally, it is not possible to simultaneously diagonalize $g$ and $W$ in our case, whereas in the Monge-Amp\`ere case these are both the usual Hessian and can be diagonalized accordingly. Therefore, to get a helpful sign on the third order terms, we are required to do some more involved linear algebra. This is located in Section \ref{subsec:third-fourth}.

    \subsection{Proof of Theorem \ref{thm:main}}	
	We begin by stating a few useful formulas for the linearized operator applied to various functions appearing in $Q$.
	\begin{align}\label{eq:L-of}
	 	&Lu = n_+ - n_-, \nonumber\\
		&Lu_P = f_P, \\
		&L\hatu(x,y) = n - y_\ga f_\ga, \nonumber\\
		&Lu_{PQ} = u^{ar}u^{bs}u_{abP}u_{rs Q} - u^{\ga\gg}u^{\gd\gb}u_{\ga\gb P} u_{\gg\gd Q} + f_{PQ}. \nonumber
	\end{align}

    Now, we differentiate the objective function twice, and since we may diagonalize $g$ at a maximum of $Q$, it suffices to compute with repeated indices
    \begin{align}
        Q_A =&\ \gs\frac{\rho_A}{\rho} + \frac{\phi_A}{\phi} + \del_A \log v^TWv \label{eq:grad}\\
        Q_{AA} =&\ \gs \left[\frac{\rho_{AA}}{\rho} - \frac{\rho_A^2}{\rho^2}\right] + \left[\frac{\phi_{AA}}{\phi} - \frac{\phi_A^2}{\phi^2}\right] + \del^2_A \log v^TWv.\label{eq:hess}
    \end{align}
    
    As $\rho = r^2 - \hatu(x,y)$, we have
 	$$LQ = \gs \left[-\frac{n - y_\ga f_\ga}{\rho} - \frac{|d\rho|_g^2}{\rho^2}\right] + g^{AA}\left[\frac{\phi_{AA}}{\phi} - \frac{\phi_A^2}{\phi^2}\right] + L \log v^TWv.$$
    
    We will ultimately be using the critical point condition to cancel the bad $|d\rho|_g^2$ terms. To do this, we first need to expand and polarize $|d\rho|_g^2$ in the Legendre transform coordinates. This will be done by using the Legendre transform map and the conjugacy of $g$ and $W$, more directly
    \begin{equation}\label{eq:coord-rewrite}
        d\rho(x,y) = -\Phi^* d\hatu(x,p).
    \end{equation}

    Working at a maximum of $Q$, we know that $v$ is a unit-Euclidean-length eigenvector of $W$ corresponding to the maximal eigenvalue $\gl$. Define $P = I - vv^T$ to be the $W,\, \gd$-orthogonal projection onto $v^\perp$ (via the Spectral Theorem). One can decompose the norm of $d\rho$ as
    \begin{equation} \label{eq:orthog-decomp}
        |d\rho|_g^2 = \frac{1}{\gl}(i_v d \hatu(x,p))^2 + \left|Pd\hatu(x,p)\right|^2_{W}
    \end{equation}

    Plugging the decomposition (\ref{eq:orthog-decomp}) into $LQ$ and reorganizing, we find
    \begin{align*}
         LQ =&\ \gs \left[-\frac{n - y_\ga f_\ga}{\rho} - \frac{1}{\rho^2 \gl}\left(i_v d \hatu(x,p)\right)^2\right] + g^{AA}\left[\frac{\phi_{AA}}{\phi} - \frac{\phi_A^2}{\phi^2}\right]\\
        &\ + \frac{L(v^TWv)}{\gl} - |d\log v^TWv|_g^2 - \frac{\gs}{\rho^2}\left|Pd\hatu(x,p)\right|^2_{W}.
    \end{align*}

    At a maximum, $dQ = 0$, so that
    $$0 = \gs\frac{d\rho}{\rho} + \frac{d\phi}{\phi} + d \log v^TWv.$$
    Composing with $P\Psi^*$ on the left gives 
    $$0 = -\frac{\gs}{\rho}Pd\hatu(x,p) + P\Psi^*\frac{d\phi}{\phi} + P\Psi^* d\log v^TWv.$$
    Rearranging and taking the one-form norm using $W^{-1}$, we may apply Young's inequality to obtain 
    $$\frac{\gs}{\rho^2}\left|Pd\hatu(x,p)\right|^2_{W} \leq  \frac{2}{\gs}\left|P\Psi^*\frac{d\phi}{\phi}\right|_{W}^2 + \frac{2}{\gs}|P\Psi^* d\log v^TWv|_{W}^2.$$
    Since $P$ is a $W$-orthogonal projection, we know that 
    $$\left|P\Psi^*\frac{d\phi}{\phi}\right|_{W}^2 \leq \left|\Psi^*\frac{d\phi}{\phi}\right|_{W}^2 = \frac{|d\phi|_g^2}{\phi^2}.$$
    
    Plugging this in yields
    \begin{align*}
        0\geq LQ \geq &\ \gs \left[-\frac{n - y_\ga f_\ga}{\rho} - \frac{1}{\rho^2 \gl}(i_v d \hatu(x,p))^2\right] + g^{AA}\left[\frac{\phi_{AA}}{\phi} - (1 + \frac{2}{\gs})\frac{\phi_A^2}{\phi^2}\right]\\
        &\ + \frac{L(v^TWv)}{\gl} - |d\log v^TWv|_g^2 - \frac{2}{\gs}|P\Psi^* d\log v^TWv|_{W}^2\\
        = &\ I_1 + I_2 + I_3
    \end{align*}
    This can be simplified further using the estimate
    \begin{equation}
        I_3 \geq -\frac{\|D_x^2f\|_{C^0}}{\gl} - 2\|D_{xy}^2f\|_{C^0} - \|D_y^2f\|_{C^0}\gl \label{eq:third-fourth}
    \end{equation}
    which will be shown in \S \ref{subsec:third-fourth} provided $\gs \geq 2$.
    
    The $\phi = \phi(\frac{1}{2}(|D_xu|^2 + |y|^2))$ terms give 
    \begin{align*}
        \phi_A =&\ \phi'(u_i u_{iA} + y_\ga \gd_{\ga A}),\\
        \phi_{AA} =&\ \phi''(u_i u_{iA} + y_\ga \gd_{\ga A})(u_j u_{jA} + y_\gb \gd_{\gb A})\\
        &\ + \phi'(u_{iA}^2 + u_iu_{iAA} + \gd_{\ga A}).
    \end{align*}
    Therefore, 
    \begin{align*}
        I_2 =& g^{AA}\left[\left(\frac{\phi''}{\phi} - (1 + \frac{2}{\gs})\left(\frac{\phi'}{\phi}\right)^2\right)(u_i u_{iA} + y_\ga \gd_{\ga A})(u_j u_{jA} + y_\gb \gd_{\gb A})\right]\\
        &\ + g^{AA}\frac{\phi'}{\phi}(u_{iA}^2 + u_iu_{iAA} + \gd_{\ga A}) \\
        =& g^{AA}\left[\left(\frac{\phi''}{\phi} - (1 + \frac{2}{\gs})\left(\frac{\phi'}{\phi}\right)^2\right)(u_i u_{iA} + y_\ga \gd_{\ga A})(u_j u_{jA} + y_\gb \gd_{\gb A})\right]\\
        &\ + \frac{\phi'}{\phi}\left(u_{ii} - u^{\ga\ga}u_{i\ga}^2 - u^{\ga\ga} + u_i f_i\right) \\
        \geq &\ \frac{\phi'}{\phi}(\tr W + u_i f_i)
    \end{align*}   
    Where the last inequality follows from choosing $\phi(t) = (1-t/M)^{-s}$ where $M = \sup_{S_r(x_0,y_0)} (|D_xu|^2 + |y|^2)$ and $s = \frac{\gs}{2}$ is to be determined. Moreover, notice that 
    \begin{align*}
        1\leq &\ \phi = (1-t/M)^{-s}\leq 2^{\frac{\gs}{2}},\\
        \frac{\gs}{2M} \leq &\ \frac{\phi'}{\phi} = \frac{s}{M}(1-\frac{t}{M})^{-1} \leq \frac{\gs}{M}.
    \end{align*}
    So that 
    $$I_2 \geq \frac{\gs}{2M}\gl - \frac{\gs}{M}\|D_xu\|_{C^0(S_r)}\|D_xf\|_{C^0(S_r)}.$$
    
    This leaves us with
    \begin{align*}
        0\geq &\  \left(\frac{\gs}{2M}  - \|D_y^2f\|_{C^0}\right)\gl\\
        &\ - \left(\frac{\gs (n + \|y\|_{C^0}\|D_yf\|_{C^0})}{\rho} + \frac{\gs}{M}\|D_xu\|_{C^0}\|D_xf\|_{C^0} + 2\|D_{xy}^2f\|_{C^0}\right)\\
        &\ - \left(\frac{\gs (i_v d \hatu(x,p))^2}{\rho^2} + \|D_x^2f\|_{C^0}\right)\gl^{-1}.
    \end{align*}
    
    Furthermore, applying (\ref{eq:xp-differential}) and the Cauchy-Schwarz inequality, we obtain
    $$|i_v d\hatu(x,p)| = |(d\hatu(x,p))(v)| \leq |v|_{\gd} |d\hatu(x,p)|_{\gd} = \sqrt{(u_i)^2 + (y_\ga)^2} \leq \sqrt{M}.$$
    So that the expression becomes
    \begin{align*}
        0\geq &\  \left(\frac{\gs}{2M}  - \|D_y^2f\|_{C^0}\right)\gl\\
        &\ - \left(\frac{\gs (n + \|y\|_{C^0}\|D_yf\|_{C^0})}{\rho} + \frac{\gs}{M}\|D_xu\|_{C^0}\|D_xf\|_{C^0} + 2\|D_{xy}^2f\|_{C^0}\right)\\
        &\ - \left(\frac{\gs M}{\rho^2} + \|D_x^2f\|_{C^0}\right)\gl^{-1}.
    \end{align*}
    We want this expression to be in terms of the objective quantity, so we may multiply through by $\rho^{2\gs}\gl$ and use the defining property of a section to obtain
    \begin{align*}
        0\geq &\  \left(\frac{\gs}{2M}  - \|D_y^2f\|_{C^0}\right)(\rho^{\gs}\gl)^2\\
        &\ - \left(\frac{\gs (n + \|y\|_{C^0}\|D_yf\|_{C^0})}{\rho} + \frac{\gs}{M}\|D_xu\|_{C^0}\|D_xf\|_{C^0} + 2\|D_{xy}^2f\|_{C^0}\right)\rho^{2\gs}\gl\\
        &\ - \left(\frac{\gs M}{\rho^2} + \|D_x^2f\|_{C^0}\right)\rho^{2\gs}.
    \end{align*}
    
    Thus, by Young's inequality with parameter and $\gs \geq 2$, we have 
    \begin{align*}
        0\geq &\  \left(\frac{\gs}{4M}  - \|D_y^2f\|_{C^0}\right)(\rho^{\gs}\gl)^2\\
        &\ - 2\gs M (n + \|y\|_{C^0}\|D_yf\|_{C^0})^2\rho^{2(\gs - 1)} - 2\frac{M}{\gs}\left(\frac{\gs}{M}\|D_xu\|_{C^0}\|D_xf\|_{C^0} + 2\|D_{xy}^2f\|_{C^0}\right)^2\rho^{2\gs}\\
        &\ - \left(\frac{\gs M}{\rho^2} + \|D_x^2f\|_{C^0}\right)\rho^{2\gs}.
    \end{align*}
    
    We define $\gs$, the scale invariant exponent by
    \begin{equation*}
        \gs = 2 + 4M\|D_y^2f\|_{C^0(S_r)}.
    \end{equation*}
    This allows us to define two  constants, $C_1$ which is $-2$-homogeneous with respect to elliptic rescalings and $C_2$ which is $+2$-homogeneous
    \begin{align*}
        C_1 =&\ [2(n + \|y\|_{C^0(S_r)}\|D_yf\|_{C^0(S_r)})^2 + 1]M\gs,\\
        C_2 =&\ 2\frac{M}{\gs}\left(\frac{\gs}{M}\|D_xu\|_{C^0(S_r)}\|D_xf\|_{C^0(S_r)} + 2\|D_{xy}^2f\|_{C^0(S_r)}\right)^2 + \|D_x^2f\|_{C^0(S_r)}.
    \end{align*}
    This yields a manifestly scale invariant estimate at a maximum of $Q$
    \begin{equation*}
        \gl \leq \sqrt{2M \left(C_1\rho^{-2} + C_2\right)}.
    \end{equation*}

    We now use two facts. First, that $1\leq \phi \leq 2^{\frac{\gs}{2}}$ and second, $\hatu \geq 0$ (see Lemma \ref{lem:hatu-pos}), so that 
    \begin{equation}\label{eq:final-est}
        \sup_{S_r(0)}\rho^{\gs}\gl \leq \sup_{S_r(0)} e^Q \leq 2^{\frac{\gs}{2}}\sqrt{2M \left(C_1 + C_2r^4\right)}r^{2(\gs - 1)}.
    \end{equation}

    \begin{remark}
        The exponent $\gs$ depends on first derivatives of the solution, but is itself scale-invariant. For $\mu >0$, $u_{\mu} = \mu^{-2}u(\mu x,\mu y)$
        solves the PDE on the domain $\Omega_\mu = \mu^{-1}\Omega$ so that 
        $$\sup_{S^{u_{\mu}}_r(0)}(|D_xu_\mu|^2(x,y) + |y|^2) = \sup_{S^{u_{\mu}}_r(0)}(\mu^{-2}|D_xu|^2(\mu x,\mu y) + |y|^2).$$
        Moreover, if $\operatorname{ctd}_0(\cdot)$ corresponds to the operation ``take the connected componenet of zero", then  
        $$S_r^{u_\mu}(0;\Omega_\mu) = \operatorname{ctd}_0\{\hatu_{\mu} < r^2\} = \operatorname{ctd}_0\{(x,y) | \hatu(\mu x,\mu y) < \mu^2 r^2\} = \mu^{-1}(S_{\mu r}^u(0;\Omega))$$
        so that 
        $$M_\mu = \sup_{S^{u_{\mu}}_r(0)}(\mu^{-2}|D_xu|^2(\mu x,\mu y) + |y|^2) = \mu^{-2}\sup_{\mu^{-1}(S^{u}_{\mu r}(0))}(|D_xu|^2(\mu x,\mu y) + |\mu y|^2) = \mu^{-2} M.$$
        Additionally $f_\mu = f(\mu x,\mu y)$ and
        $$\|D^2_yf_\mu\|_{C^0(S_r^{u_\mu}(0))} = \|\mu^2 D^2_yf(\mu x,\mu y)\|_{C^0(\mu^{-1}(S^u_{\mu r}(0)))} = \mu^2 \|D^2_yf\|_{C^0(S_{\mu r}^u(0))}$$
        therefore, $\sigma_\mu = \sigma$.

        Also, notice that $\rho$ scales $-2$-homogeneously
        $$\rho_{r}(x,y) = r^2 - \hatu_{\mu}(x, y) = \mu^{-2}[(\mu r)^2 - \hatu (\mu x,\mu y)] = \mu^{-2}\rho_{\mu r}(x,y).$$
    \end{remark}
    
    \subsection{Estimate (\ref{eq:third-fourth})} \label{subsec:third-fourth}
    We turn our attention now to the terms 
    $$I_3 = \frac{L(v^TWv)}{v^TWv} - |d\log v^TWv|_g^2 - \frac{2}{\gs}|P\Psi^* d\log v^TWv|_{W}^2.$$
    Using the calculations from the appendices, we find that 
    \begin{align*}
        I_3 =&\ \frac{|S(v,\cdot,\cdot)|_W^2 + D^2f(\Psi_*v,\Psi_*v)}{v^TWv} - \left|\Phi^* \Psi^*\frac{d v^TWv}{v^TWv}\right|_g^2 - \frac{2}{\gs}\left|P\Psi^* \frac{dv^TWv}{v^TWv}\right|_{W}^2\\
        =&\ \frac{|S(v,\cdot,\cdot)|_W^2 + D^2f(\Psi_*v,\Psi_*v)}{v^TWv} - \left|\frac{S(v,v,\cdot)}{v^TWv}\right|_W^2 - \frac{2}{\gs}\left|\frac{S(v,v,P\cdot)}{v^TWv}\right|_{W}^2
    \end{align*}

    But notice that $i_vS$ and $i_v^2S$ can be orthogonally decomposed along $v$
    \begin{align*}
       S(v,\cdot , \cdot) =&\ S(v,v,v)v^T\otimes v^T + v^T\otimes S(v,v,P\cdot) + S(v,P\cdot, v) \otimes v^T + S(v,P\cdot,P\cdot),\\
       S(v,v , \cdot) =&\ S(v,v,v) v^T + S(v,v,P\cdot).
    \end{align*}
    Expanding the norms arising in $I_3$ via the orthogonal decompositions, we find
    \begin{equation*}
        I_3 = \left[1-\frac{2}{\gs}\right]\frac{|S(v,v,P\cdot)|^2_W}{(v^TWv)^2} + \frac{|S(v,P\cdot,P\cdot)|_W^2}{v^TWv} + \frac{D^2f(\Psi_*v,\Psi_*v)}{v^TWv}.
    \end{equation*}
    From which we obtain that 
    $$I_3 \geq \frac{D^2f(\Psi_*v,\Psi_*v)}{v^TWv}$$
    whenever $2 \leq \gs$.

    Finally, notice that $\Psi_*$ and $W$ are related by
    $$\Psi_* v = \begin{pmatrix}
        v^+\\
        -\pi_-Wv
    \end{pmatrix}.$$
    Thus, the $D^2f$ terms expand as 
    $$D^2f(\Psi_*v,\Psi_*v) = D^2f(v^+,v^+) - 2 \gl D^2f(v^+,v^-) + \gl^2 D^2f(v^-,v^-).$$

    Therefore, the $D^2f$ terms are bounded below by 
    $$I_3 \geq -\frac{\|D_x^2f\|_{C^0}}{\gl} - 2\|D_{xy}^2f\|_{C^0} - \|D_y^2f\|_{C^0}\gl.$$

 \section{Liouville Theorem (Proof of Corollary \ref{cor:conditional-JCP})}

     \begin{lemma}\label{lem:grad}
        Suppose that $u\in \mathcal{A}(\Omega)$ and $\Omega$ is convex, then for every $(x,y)\in \Omega$,
		\begin{align*}
			|D_xu|(x,y) \leq& \frac{\sup_{\Omega}u - u(x,y)}{\dist((x,y), \del \Omega)}\\
            |D_yu|(x,y) \leq& \frac{u(x,y) - \inf_{\Omega}u}{\dist((x,y), \del \Omega)} 
		\end{align*}
        \begin{proof}
            The function $v(x) = u(x,y)$ is convex for each $y$, if $(x,y),(x',y')\in \Omega$ (also convex), then if $y=y'$ we have
            $$v(x) \geq v(x') + (x-x')^T D_xv(x').$$
            Taking $x = x' + s\frac{D_xv}{|D_xv|}(x')$ then implies that 
            $$v(x) - v(x') \geq s |D_xv|(x'), \quad \forall 0\leq s < \dist((x',y'),\del \Omega).$$
            Notice that this interval is suboptimal, but it will lend itself to a symmetric gradient estimate. Expanding in terms of $u$ yields the result.

            Similarly, $w(y) = u(x,y)$ is concave for each $x$. If $(x,y),(x',y')\in \Omega$, then when $x=x'$ we have
            $$w(y) \leq w(y') + (y-y')^TD_yw(y').$$
            Taking $y = y' - s \frac{D_yw}{|D_yw|}(y')$ then implies that 
            $$w(y') - w(y) \geq s|D_yw|(y'), \forall 0\leq s < \dist((x',y'),\del \Omega).$$
        \end{proof}
    \end{lemma}

    \begin{proof}[Proof of Corollary \ref{cor:conditional-JCP}]
    By the quadratic growth hypothesis on $\hatu$, we have that 
    $$C^{-2}(|x|^2+|y|^2) \leq \hatu(x,y) \leq R^2,$$
    which implies that sections are contained in balls, i.e. $S_R(0) \subset B_{CR}(0)$. Conversely, if $(x,y) \in B_{\frac{R}{C}}(0)$, then 
    $$\hatu(x,y) \leq C^2 (|x|^2 + |y|^2) \leq R^2,$$
    showing that every section contains a ball, i.e. $B_{\frac{R}{C}}(0) \subset S_R(0)$. To summarize, the sections are all trapped by balls
    \begin{equation}\label{eq:trapping}
        B_{\frac{R}{C}}(0) \subset S_R(0) \subset B_{CR}(0).
    \end{equation}

    We now apply Lemma \ref{lem:grad} on the convex set $B_{2CR}$, the ball trapping (\ref{eq:trapping}), and take the supremum with respect to $S_{R}$ to find 
    $$\sup_{S_{R}}|D_xu|(x,y) \leq \frac{\sup_{B_{2CR}(0)}u - u(x,y)}{\dist(S_{R}, \del B_{2CR})}.$$
    By the ball-trapping inequality (\ref{eq:trapping}), we know that $S_{R} \subset B_{CR}$ and we may estimate the right-hand side using the quadratic growth condition on $u$ to obtain
    $$\frac{\sup_{B_{2CR}(0)}u - u(x,y)}{\dist(S_{R}, \del B_{2CR})} \leq \frac{5AC^2R^2}{\dist(B_{CR}, \del B_{2CR})}.$$
    However, 
    $$\dist(B_{CR}, \del B_{2CR}) = CR>0.$$
    This yields the gradient estimate
    \begin{equation}\label{eq:gradient-est}
        \sup_{S_{R}}|D_xu|(x,y) \leq C_3(A,C)R
    \end{equation}
    where $C_3$ is independent of $R$. Therefore, by (\ref{eq:gradient-est}) and (\ref{eq:trapping}), we find that 
    $$M_{R} \leq C_4(A,C)R^2.$$
    
    Applying Corollary \ref{cor:fconst} to the section $S_R(0)$, we have 
    $$\sup_{S_R(0)}(\rho^2_R |W(D^2u)|)\leq C_5(A,C,n) R^4.$$
    But by restricting to $S_{\frac{R}{2}}(0)$, we see that $\rho^2 > \frac{9}{16}R^4$, which yields
    $$\sup_{S_{\frac{R}{2}}}|W| \leq C_6(A,C,n).$$
    By ball trapping, we find 
    $$\sup_{B_{\frac{R}{2C}}}|W| \leq C_6(A,C,n).$$
    Since $R>0$ was arbitrary and $C_6$ is independent of $R$, we find that 
    \begin{equation}\label{eq:UE}
        \sup_{\R^n}|W| \leq C_6(A,C,n).
    \end{equation}
    Therefore, since $\det W = c_0$, rearranging and (\ref{eq:UE}) imply uniform ellipticity, i.e.
    $$\det W = c_0 \implies \inf_{\R^n} \gl_{min}(W) \geq \frac{c_0}{\sup_{\R^n}\gl_{max}^{n-1}} \geq \frac{c_0}{C_6^{n-1}}>0.$$
    We set 
    $$C_7(A,C,n,c_0) = \max \left\{C_6(A,C,n), \frac{C_6^{n-1}(A,C,n)}{c_0}\right\}$$

    By (\ref{def:Wmatrix}), this implies that the solution is uniformly elliptic. To see why, notice that if $C_7^{-1}I \leq W\leq C_7 I$, then
    $$C_7^{-1} I_{\R^{n_-}}\leq (W_{\ga \gb}) = - (D_y^2u)^{-1} \leq C_7 I_{\R^{n_-}}.$$
    From which we obtain,  
    \begin{equation}\label{eq:D2yubnd}
    C_7^{-1}I_{\R^{n_-}} \leq -D_y^2u \leq C_7I_{\R^{n_-}}.
    \end{equation}
    Furthermore, we have that 
    $$C_7^{-1}I_{\R^{n_+}} \leq (W_{ij}) = D_x^2u - (D_{xy}^2u)^T(D_y^2u)^{-1}(D_{xy}^2u) \leq C_7I_{\R^{n_+}}.$$
    However, by positivity of $-D_y^2u$ and $D_x^2u$, we know that
    \begin{equation}\label{eq:D2xubnd}
    0 < D_x^2u \leq (W_{ij}) \leq C_7I_{\R^{n_+}}.
    \end{equation}
    Then by combining (\ref{eq:D2yubnd}) with the equation, we find that 
    $$c_0C_7^{-n_-} \leq \det D^2_xu \leq c_0C_7^{n_-}.$$
    This along with the upper bound (\ref{eq:D2xubnd}) implies the eigenvalue lower bound
    $$\gl_{min}(D^2_xu)\geq \frac{c_0C_7^{-n_-}}{(\gl_{max}(D_x^2u))^{n_+-1}} \geq \frac{c_0}{C_7^{n-1}}.$$
    Therefore the metric $g$ satisfies 
    $$C_8^{-1}I \leq g \leq C_8I, \quad C_8(A,C,n,c_0) = \max\left\{C_7(A,C,n,c_0), \frac{C_7^{n-1}}{c_0}\right\}.$$
    Crucially, $C_8$ is independent of $R$ and (\ref{eq:TMA}) is uniformly elliptic on $\R^n$. At this point, we may invoke the Streets-Warren Liouville Theorem \cite[][Corollary 2.4]{streets_evans-krylov_2016} and conclude the proof.
    \end{proof}
    
    \section*{Acknowledgements}
    A substantial part of this work was completed during the author's time as a postdoc at the University of Iowa. The author would like to thank Hao Fang for many helpful discussions regarding this project. The author would also like to thank Bryan Dimler and Aranya Sen for their very helpful feedback on earlier versions of this manuscript.

    \section*{Declaration of Generative AI Use}
    The author has used Anthropic's Claude Opus 5 and ChatGPT 5.6 Sol for assistance in closing gaps in the proof of Lemma \ref{lem:hatu-pos}, ChatGPT 5.6 Sol for generating the counter-example to show that the bound on $|u|$ does not imply the growth hypothesis in Corollary \ref{cor:conditional-JCP}, and both models for proofreading and editing. The author takes full responsibility for the contents of this paper.

    \section*{Funding}
    This work has been supported in part by the NSF via Research Training Grant DMS-2038103.

    \section*{Conflict of Interest Statement}
    The author has no conflicts of interest to declare.

    \bibliographystyle{tfnlm}
    \bibliography{mybib2}
    
 	\appendix{\section{$W$-Calculations}}
 	We begin by remarking that there are well-known formulae \cite[cf.][]{streets_evans-krylov_2016} for $W$ and $W^{-1}$
 	\begin{align} \label{def:Wmatrix2}
 		&W_{ij} = u_{ij} - u_{i\gg}u^{\gg\gd}u_{j\gd}, &\ &W^{ij} = u^{ij}, \nonumber\\
 		&W_{i\ga} = u_{i\gg}u^{\gg\ga}, &\ &W^{i\ga} = u^{i p}u_{p \ga},\\
 		&W_{\ga\gb} = - u^{\ga\gb}, &\ &W^{\ga\gb} = -u_{\ga\gb} + u_{\ga k}u^{kl}u_{l\gb}.\nonumber
 	\end{align}
 	In these formulas, it is important to emphasize that upper and lower indices on $u$ and $W$ are being used differently. Namely, upper indices on $u$ refer to the inverses of the partial Hessians $D_{x}^2u$ and $D_y^2u$, whereas upper indices on $W$ refer to components of the full inverse of $W$.
 	
 	\begin{lemma}\label{lem:LW-from-u}
 		Suppose that $u\in C^4$ is a solution of (\ref{eq:TMA}), then the matrix $W$ defined in (\ref{def:Wmatrix2}) is acted on by the linearized operator $L$ at a point $(x,y)$ where $g$ has been diagonalized in the following way.
 		\begin{align*}
 			LW_{ij} = &\ \frac{1}{u_{kk}u_{ll}}\left(u_{jkl} - \frac{u_{j\gb}u_{kl\gb}}{u_{\gb\gb}}\right)\left(u_{ikl} - \frac{u_{i\ga}u_{kl\ga}}{u_{\ga\ga}}\right) - \frac{2}{u_{kk}u_{\gg\gg}} \left(u_{jk\gg} - \frac{u_{j\gb}u_{k\gb\gg}}{u_{\gb\gb}}\right)\left(u_{ik\gg} - \frac{u_{i\ga}u_{k\ga\gg}}{u_{\ga\ga}}\right)\\
 			&\ + \frac{1}{u_{\gg\gg}u_{\gd\gd}}\left(u_{j\gg\gd} - \frac{u_{j\gb}u_{\gb\gg\gd}}{u_{\gb\gb}}\right)\left(u_{i\gg\gd} - \frac{u_{i\ga}u_{\ga\gg\gd}}{u_{\ga\ga}}\right) + \left(f_{ij} - \frac{u_{j\gg}f_{i\gg}}{u_{\gg\gg}} - \frac{u_{i\gg}f_{j\gg}}{u_{\gg\gg}} + \frac{u_{i\gg}f_{\gg\gd}u_{j\gd}}{u_{\gg\gg}u_{\gd\gd}}\right)\\
 			LW_{i\ga} =&\ \frac{1}{u_{\ga\ga}}\frac{u_{kl\ga}}{u_{kk}u_{ll}}\left( u_{ikl} - \frac{u_{i\gb}u_{kl\gb}}{u_{\gb\gb}}\right) - \frac{2}{u_{\ga\ga}}\frac{u_{k\ga\gg}}{u_{kk}u_{\gg\gg}}\left( u_{ik\gg}  - \frac{u_{i\gb}u_{k\gb\gg}}{u_{\gb\gb}}\right)\\
 			&\ + \frac{1}{u_{\ga\ga}}\frac{u_{\ga\gg\gd}}{u_{\gg\gg}u_{\gd\gd}}\left(u_{i\gg\gd}  - \frac{u_{i\gb}u_{\gb\gg\gd}}{u_{\gb\gb}}\right) + \left(\frac{f_{i\ga}}{u_{\ga\ga}} - \frac{u_{i\gg}f_{\gg\ga}}{u_{\gg\gg}u_{\ga\ga}}\right)\\
 			LW_{\ga\gb}= &\ \frac{1}{u_{\ga\ga}u_{\gb\gb}}\left(\frac{u_{kl\ga}u_{kl\gb}}{u_{kk}u_{ll}} - 2 \frac{u_{k\ga\gg}u_{k\gb\gg}}{u_{kk}u_{\gg\gg}} + \frac{u_{\ga\gg\gd}u_{\gb\gg\gd}}{u_{\gg\gg}u_{\gd\gd}}\right) + \left(\frac{f_{\ga\gb}}{u_{\ga\ga}u_{\gb\gb}}\right)
 		\end{align*}
 		\begin{proof}
 			\begin{align*}
 				W_{ij,A} =&\ u_{ijA} - u_{i\gg A}u^{\gg\gd}u_{j\gd} - u_{i\gg}u^{\gg\gd}u_{j\gd A} + u_{i\gg}u^{\gg\mu}u_{\mu\nu A}u^{\gd \nu}u_{j\gd}\\
 				W_{i\ga, A} =&\ u_{i\gg A}u^{\gg \ga} - u_{i\gg}u^{\gg \mu}u_{\mu \nu A}u^{\nu \ga}\\
 				W_{\ga\gb, A}=&\ u^{\ga \gg}u_{\gg \gd A}u^{\gd \gb}
 			\end{align*}
 			
 			The second derivatives give 
 			\begin{align*}
 				W_{ij,AA} =&\ u_{ijAA} - u_{i\gg AA}u^{\gg\gd}u_{j\gd} - u_{i\gg}u^{\gg\gd}u_{j\gd AA} + u_{i\gg}u^{\gg\mu}u_{\mu\nu AA}u^{\gd \nu}u_{j\gd}\\
 				&\ - 2u_{i\gg A}u^{\gg\gd}u_{j\gd A} + 2u_{i\gg A}u^{\gg\mu}u_{\mu\nu A}u^{\nu\gd}u_{j\gd} + 2u_{i\gg}u^{\gg\mu}u_{\mu\nu A}u^{\nu\gd}u_{j\gd A}\\
 				&\ -2 u_{i\gg}u^{\gg\rho}u_{\rho \gl A} u^{\gl\mu}u_{\mu\nu A}u^{\gd \nu}u_{j\gd}\\
 				W_{i\ga, AA} =&\ u_{i\gg AA}u^{\gg \ga}- u_{i\gg}u^{\gg \mu}u_{\mu \nu AA}u^{\nu \ga} - 2u_{i\gg A}u^{\gg\mu}u_{\mu\nu A}u^{\nu \ga} + 2u_{i\gg}u^{\gg \gd}u_{\gd\ep A}u^{\ep \mu}u_{\mu \nu A}u^{\nu \ga}\\
 				W_{\ga\gb,AA}=&\ u^{\ga \gg}u_{\gg \gd A A}u^{\gd \gb} - 2u^{\ga\mu}u_{\mu\nu A}u^{\nu \gg}u_{\gg \gd A}u^{\gd \gb}
 			\end{align*}
 			
 			At a diagonal point
 			\begin{align*}
 				W_{ij,A} =&\ u_{ijA} - \frac{u_{i\gg A}u_{j\gg}}{u_{\gg\gg}} - \frac{u_{i\gg}u_{j\gg A}}{u_{\gg\gg}} + \frac{u_{i\gg}u_{\gg\gd A}u_{j\gd}}{u_{\gg\gg}u_{\gd\gd}}\\
 				W_{i\ga, A} =&\ \frac{u_{i\ga A}}{u_{\ga \ga}} - \frac{u_{i\gg}u_{\gg \ga A}}{u_{\ga\ga}u_{\gg\gg}}\\
 				W_{\ga\gb, A}=&\ \frac{u_{\ga \gb A}}{u_{\ga\ga}u_{\gb\gb}}
 			\end{align*}
 			
 			\begin{align*}
 				W_{ij,AA} =&\ u_{ijAA} - \frac{u_{i\gg AA}u_{j\gg}}{u_{\gg\gg}} - \frac{u_{i\gg}u_{j\gg AA}}{u_{\gg\gg}} + \frac{u_{i\gg}u_{\gg\gd AA}u_{j\gd}}{u_{\gg\gg}u_{\gd\gd}}\\
 				&\ - 2\frac{u_{i\gg A}u_{j\gg A}}{u_{\gg\gg}} + 2\frac{u_{i\gg A}u_{\gg\gd A}u_{j\gd}}{u_{\gg\gg}u_{\gd\gd}} + 2\frac{u_{i\gg}u_{\gg\gd A}u_{j\gd A}}{u_{\gg\gg}u_{\gd\gd}} -2 \frac{u_{i\gg}u_{\gg \mu A}u_{\mu\gd A}u_{j\gd}}{u_{\gg\gg}u_{\gd\gd}u_{\mu\mu}}\\
 				W_{i\ga, AA} =&\ \frac{u_{i\ga AA}}{u_{\ga\ga}} - \frac{u_{i\gg}u_{\gg \ga AA}}{u_{\gg\gg}u_{\ga\ga}} - 2\frac{u_{i\gg A}u_{\gg\ga A}}{u_{\gg\gg}u_{\ga\ga}} + 2\frac{u_{i\gg}u_{\gg\gd A}u_{\gd \ga A}}{u_{\gg\gg}u_{\gd\gd}u_{\ga\ga}}\\
 				W_{\ga\gb,AA}=&\ \frac{u_{\ga \gb A A}}{u_{\ga\ga}u_{\gb\gb}} - 2\frac{u_{\ga\gg A}u_{\gg \gb A}}{u_{\ga\ga}u_{\gb\gb}u_{\gg\gg}}
 			\end{align*}
 			
 			Tracing by $g^{-1}$ gives 
 			\begin{align*}
 				LW_{ij} =&\ \left(Lu_{ij} - 2\frac{u_{i\gg A}g^{AA}u_{j\gg A}}{u_{\gg\gg}}\right) - \frac{u_{j\gg}}{u_{\gg\gg}}\left(Lu_{i\gg} - 2\frac{u_{i\gd A}g^{AA}u_{\gg\gd A}}{u_{\gd\gd}}\right)\\
 				&\ - \frac{u_{i\gg}}{u_{\gg\gg}}\left(Lu_{j\gg} - 2\frac{u_{\gg\gd A}g^{AA}u_{j\gd A}}{u_{\gd\gd}}\right) + \frac{u_{i\gg}u_{j\gd}}{u_{\gg\gg}u_{\gd\gd}}\left(Lu_{\gg\gd} -2 \frac{u_{\gg \mu A}g^{AA}u_{\mu\gd A}}{u_{\mu\mu}}\right)\\
 				LW_{i\ga} =&\ \frac{1}{u_{\ga\ga}}\left(Lu_{i\ga} - 2\frac{u_{i\gg A}g^{AA}u_{\gg\ga A}}{u_{\gg\gg}}\right) - \frac{u_{i\gg}}{u_{\gg\gg}u_{\ga\ga}}\left(Lu_{\gg \ga} - 2\frac{u_{\gg\gd A}g^{AA}u_{\gd \ga A}}{u_{\gd\gd}}\right)\\
 				LW_{\ga\gb}=&\ \frac{1}{u_{\ga\ga}u_{\gb\gb}}\left(Lu_{\ga \gb}  - 2\frac{u_{\ga\gg A}g^{AA}u_{\gg \gb A}}{u_{\gg\gg}}\right)
 			\end{align*}
 			From which, the result follows by (\ref{eq:L-of}).
 		\end{proof}
 	\end{lemma}
 	
 	The formula for $LW$ contained in Lemma \ref{lem:LW-from-u} will be much more convenient to work with if it is written in terms of derivatives of $W$ instead. Therefore, we will prove the following theorem, motivated by analogy to the Monge-Amp\`ere equation.
 	\begin{proposition} \label{thm:LW-from-W}
 		Suppose that $u\in C^4$ is a solution of (\ref{eq:TMA}), then the matrix $W$ defined in (\ref{def:Wmatrix2}) is acted on by the linearized operator $L$ at a point $(x,y)$ where $g$ has been diagonalized according to the formula 
 		\begin{equation}\label{eq:LW-from-W}
 			LW_{AB} = W^{PQ}g^{RS}W_{AP,R}W_{BQ,S} + (\Psi_*)_A^Pf_{PQ}(\Psi_*)^Q_B.
 		\end{equation}
 		
 		\begin{proof}
 			Fix a point $(x,y)$ and diagonalize $g(x,y)$. One can then compute
 			\begin{align*}
 				W^{PQ}g^{RS}W_{AP,R}W_{BQ,S} =&\ W^{ij}\left(\frac{W_{Ai,k}W_{Bj,k}}{u_{kk}}  - \frac{W_{Ai,\mu}W_{Bj,\mu}}{u_{\mu\mu}}\right) + W^{\nu \rho}\left(\frac{W_{A\nu,k}W_{B\rho,k}}{u_{kk}} - \frac{W_{A\nu,\mu}W_{B\rho,\mu}}{u_{\mu\mu}}\right)\\
 				&\ + W^{i\nu}\left(\frac{(W_{Ai,k}W_{B\nu,k} + W_{A\nu,k}W_{B i,k})}{u_{kk}} - \frac{(W_{Ai,\mu}W_{B\nu,\mu} + W_{A\nu,\mu}W_{Bi,\mu})}{u_{\mu\mu}}\right)\\
 				=&\ u^{ii}\left(\frac{W_{Ai,k}W_{Bi,k}}{u_{kk}}  - \frac{W_{Ai,\mu}W_{Bi,\mu}}{u_{\mu\mu}}\right)\\
 				&\ + \left(-u_{\nu\rho} + \frac{u_{\nu p}u_{p\rho}}{u_{pp}}\right)\left(\frac{W_{A\nu,k}W_{B\rho,k}}{u_{kk}} - \frac{W_{A\nu,\mu}W_{B\rho,\mu}}{u_{\mu\mu}}\right)\\
 				&\ + \left(u^{i p}u_{p \nu}\right)\left(\frac{(W_{Ai,k}W_{B\nu,k} + W_{A\nu,k}W_{B i,k})}{u_{kk}} - \frac{(W_{Ai,\mu}W_{B\nu,\mu} + W_{A\nu,\mu}W_{Bi,\mu})}{u_{\mu\mu}}\right)
 			\end{align*}
 			
 			We will prove Theorem \ref{thm:LW-from-W} arguing type-by-type in the next 3 cases.
 			\begin{description} 	
 				\item[Case 1 $(-/-)$]
 					In the setting in which $A = \ga$ and $B = \gb$, we can expand the right-hand side of (\ref{eq:LW-from-W}) as
 					\begin{align*}
 						W^{PQ}g^{RS}W_{\ga P,R}W_{\gb Q,S} =&\ u^{ii}\left(\frac{W_{\ga i,k}W_{\gb i,k}}{u_{kk}}  - \frac{W_{\ga i,\mu}W_{\gb i,\mu}}{u_{\mu\mu}}\right)\\
 						&\ + \left(-u_{\nu\rho} + \frac{u_{\nu p}u_{p\rho}}{u_{pp}}\right)\left(\frac{W_{\ga \nu,k}W_{\gb \rho,k}}{u_{kk}} - \frac{W_{\ga\nu,\mu}W_{\gb\rho,\mu}}{u_{\mu\mu}}\right)\\
 						&\ + \left(u^{i p}u_{p \nu}\right)\left(\frac{(W_{\ga i,k}W_{\gb\nu,k} + W_{\ga\nu,k}W_{\gb i,k})}{u_{kk}} - \frac{(W_{\ga i,\mu}W_{\gb\nu,\mu} + W_{\ga\nu,\mu}W_{\gb i,\mu})}{u_{\mu\mu}}\right).
 					\end{align*}
 			
 					It will then be easiest to go line by line. Taking the first line, we have
 					\begin{align*}
 						u^{ii}\left(\frac{W_{\ga i,k}W_{\gb i,k}}{u_{kk}}  - \frac{W_{\ga i,\mu}W_{\gb i,\mu}}{u_{\mu\mu}}\right) =&\ \frac{1}{u_{ii}u_{kk}}\left(\frac{u_{ik\ga}}{u_{\ga \ga}} - \frac{u_{i\gg}u_{k\ga\gg}}{u_{\ga\ga}u_{\gg\gg}}\right)\left(\frac{u_{ik\gb}}{u_{\gb \gb}} - \frac{u_{i\gd}u_{k\gb\gd}}{u_{\gb\gb}u_{\gd\gd}}\right)\\
 						&\  - \frac{1}{u_{ii}u_{\mu\mu}}\left(\frac{u_{i\ga \mu}}{u_{\ga \ga}} - \frac{u_{i\gg}u_{\ga\gg\mu}}{u_{\ga\ga}u_{\gg\gg}}\right)\left(\frac{u_{i\gb \mu}}{u_{\gb \gb}} - \frac{u_{i\gd}u_{\gb\gd\mu}}{u_{\gb\gb}u_{\gd\gd}}\right).
 					\end{align*}
 					The second line is then
 					\begin{align*}
 						\left(-u_{\nu\rho} + \frac{u_{\nu p}u_{p\rho}}{u_{pp}}\right) &\ \left(\frac{W_{\ga \nu,k}W_{\gb \rho,k}}{u_{kk}} - \frac{W_{\ga\nu,\mu}W_{\gb\rho,\mu}}{u_{\mu\mu}}\right)\\
 						= &\ \frac{1}{u_{kk}}\left(-u_{\nu\rho} + \frac{u_{\nu p}u_{p\rho}}{u_{pp}}\right)\left(\frac{u_{k\ga \nu}}{u_{\ga\ga}u_{\nu\nu}}\right)\left(\frac{u_{k\gb\rho}}{u_{\gb\gb}u_{\rho\rho}}\right) \\
 						&\ - \frac{1}{u_{\mu\mu}}\left(-u_{\nu\rho} + \frac{u_{\nu p}u_{p\rho}}{u_{pp}}\right)\left(\frac{u_{\ga \nu \mu}}{u_{\ga\ga}u_{\nu\nu}}\right)\left(\frac{u_{\gb\rho \mu}}{u_{\gb\gb}u_{\rho\rho}}\right).
 					\end{align*}
 					The third line is then
 					\begin{align*}
 						\left(u^{i p}u_{p \nu}\right)& \left(\frac{(W_{\ga i,k}W_{\gb\nu,k} + W_{\ga\nu,k}W_{\gb i,k})}{u_{kk}} - \frac{(W_{\ga i,\mu}W_{\gb\nu,\mu} + W_{\ga\nu,\mu}W_{\gb i,\mu})}{u_{\mu\mu}}\right)\\
 						=&\  \frac{u_{i \nu}}{u_{ii}u_{kk}}\left[\left(\frac{u_{ik\ga}}{u_{\ga \ga}} - \frac{u_{i\gg}u_{k\ga\gg}}{u_{\ga\ga}u_{\gg\gg}}\right)\left(\frac{u_{k\gb\nu}}{u_{\gb\gb}u_{\nu\nu}}\right) + \left( \frac{u_{k\ga \nu}}{u_{\ga\ga}u_{\nu\nu}}\right)\left( \frac{u_{ik\gb}}{u_{\gb \gb}} - \frac{u_{i\gd}u_{k\gb\gd}}{u_{\gb\gb}u_{\gd\gd}}\right)\right]\\
 						&\ - \frac{u_{i \nu}}{u_{ii}u_{\mu\mu}} \left[\left(\frac{u_{i\ga\mu}}{u_{\ga \ga}} - \frac{u_{i\gg}u_{\ga\gg\mu}}{u_{\ga\ga}u_{\gg\gg}}\right)\left(\frac{u_{\gb\nu\mu}}{u_{\gb\gb}u_{\nu\nu}}\right) + \left( \frac{u_{\ga \nu\mu}}{u_{\ga\ga}u_{\nu\nu}}\right)\left( \frac{u_{i\gb\mu}}{u_{\gb \gb}} - \frac{u_{i\gd}u_{\gb\gd\mu}}{u_{\gb\gb}u_{\gd\gd}}\right)\right].
 					\end{align*}
 			
 					Combining, simplifying, and comparing to Lemma \ref{lem:LW-from-u} then yields 
 					$$W^{PQ}g^{RS}W_{\ga P,R}W_{\gb Q,S} + (\Psi_*)_\ga^Pf_{PQ}(\Psi_*)^Q_\gb = LW_{\ga\gb}.$$
 			
 				\item[Case 2 $(+/)$]
 					Proceeding to $A=j$ and $B=\gb$, we can expand the right-hand side of (\ref{eq:LW-from-W}) as
 					\begin{align*}
 						W^{PQ}g^{RS}W_{jP,R}W_{\gb Q,S} =&\ u^{ii}\left(\frac{W_{ij,k}W_{i\gb,k}}{u_{kk}}  - \frac{W_{ij,\mu}W_{i\gb,\mu}}{u_{\mu\mu}}\right)\\
 						&\ + \left(-u_{\nu\rho} + \frac{u_{\nu p}u_{p\rho}}{u_{pp}}\right)\left(\frac{W_{j\nu,k}W_{\gb\rho,k}}{u_{kk}} - \frac{W_{j\nu,\mu}W_{\gb\rho,\mu}}{u_{\mu\mu}}\right)\\
 						&\ + \left(u^{i p}u_{p \nu}\right)\left(\frac{(W_{ij,k}W_{\gb\nu,k} + W_{j\nu,k}W_{i\gb,k})}{u_{kk}} - \frac{(W_{ij,\mu}W_{\gb\nu,\mu} + W_{j\nu,\mu}W_{i\gb,\mu})}{u_{\mu\mu}}\right).
 					\end{align*}
 			
 					Taking the first line, we have 
 					\begin{align*}
 						u^{ii}&\ \left(\frac{W_{ij,k}W_{i\gb,k}}{u_{kk}}  - \frac{W_{ij,\mu}W_{i\gb,\mu}}{u_{\mu\mu}}\right)\\
 						=&\ \frac{1}{u_{ii}u_{kk}}\left(u_{ijk} - \frac{u_{ik\gg}u_{j\gg}}{u_{\gg\gg}} - \frac{u_{i\gg}u_{jk\gg}}{u_{\gg\gg}} + \frac{u_{i\gg}u_{k\gg\gd}u_{j\gd}}{u_{\gg\gg}u_{\gd\gd}}\right)\left(\frac{u_{ik\gb}}{u_{\gb \gb}} - \frac{u_{i\nu}u_{k\gb\nu}}{u_{\gb\gb}u_{\nu\nu}}\right)\\
 						&\ - \frac{1}{u_{ii}u_{\mu\mu}}\left(u_{ij\mu} - \frac{u_{i\gg \mu}u_{j\gg}}{u_{\gg\gg}} - \frac{u_{i\gg}u_{j\gg \mu}}{u_{\gg\gg}} + \frac{u_{i\gg}u_{\gg\gd \mu}u_{j\gd}}{u_{\gg\gg}u_{\gd\gd}}\right)\left(\frac{u_{i\gb \mu}}{u_{\gb \gb}} - \frac{u_{i\nu}u_{\gb\nu \mu}}{u_{\gb\gb}u_{\nu\nu}}\right).
 					\end{align*}
 					The second line is
 					\begin{align*}
 						\left(-u_{\nu\rho} + \frac{u_{\nu p}u_{p\rho}}{u_{pp}}\right)&\ \left(\frac{W_{j\nu,k}W_{\gb\rho,k}}{u_{kk}} - \frac{W_{j\nu,\mu}W_{\gb\rho,\mu}}{u_{\mu\mu}}\right)\\
 						=&\ \frac{1}{u_{kk}}\left(-u_{\nu\rho} + \frac{u_{\nu p}u_{p\rho}}{u_{pp}}\right)\left(\frac{u_{jk\nu}}{u_{\nu \nu}} - \frac{u_{j\gg}u_{k\gg \nu}}{u_{\nu\nu}u_{\gg\gg}}\right)\left(\frac{u_{k\gb \rho}}{u_{\gb\gb}u_{\rho\rho}}\right) \\
 						&\ - \left(\frac{1}{u_{\mu\mu}}\right)\left(-u_{\nu\rho} + \frac{u_{\nu p}u_{p\rho}}{u_{pp}}\right) \left(\frac{u_{j\nu \mu}}{u_{\nu \nu}} - \frac{u_{j\gg}u_{\gg \nu \mu}}{u_{\nu\nu}u_{\gg\gg}}\right)\left(\frac{u_{\gb \rho \mu}}{u_{\gb\gb}u_{\rho\rho}}\right).
 					\end{align*}
 					The third line is 
 					\begin{align*}
 						\left(u^{i p}u_{p \nu}\right)&\ \left(\frac{(W_{ij,k}W_{\gb\nu,k} + W_{j\nu,k}W_{i\gb,k})}{u_{kk}} - \frac{(W_{ij,\mu}W_{\gb\nu,\mu} + W_{j\nu,\mu}W_{i\gb,\mu})}{u_{\mu\mu}}\right)\\
 						=&\ \frac{u_{i \nu}}{u_{ii}u_{kk}} \left(u_{ijk} - \frac{u_{ik\gg}u_{j\gg}}{u_{\gg\gg}} - \frac{u_{i\gg}u_{jk\gg}}{u_{\gg\gg}} + \frac{u_{i\gg}u_{k\gg\gd}u_{j\gd}}{u_{\gg\gg}u_{\gd\gd}}\right) \left(\frac{u_{k\gb\nu}}{u_{\gb\gb}u_{\nu\nu}}\right)\\
 						&\ + \frac{u_{i \nu}}{u_{ii}u_{kk}}\left(\frac{u_{jk\nu}}{u_{\nu \nu}} - \frac{u_{j\gd}u_{k\gd \nu}}{u_{\nu\nu}u_{\gd\gd}}\right) \left(\frac{u_{ik\gb}}{u_{\gb \gb}} - \frac{u_{i\gg}u_{k\gg\gb}}{u_{\gb\gb}u_{\gg\gg}}\right)\\
 						&\ - \frac{u_{i \nu}}{u_{ii}u_{\mu\mu}} \left(u_{ij\mu} - \frac{u_{i\gg\mu}u_{j\gg}}{u_{\gg\gg}} - \frac{u_{i\gg}u_{j\gg \mu}}{u_{\gg\gg}} + \frac{u_{i\gg}u_{\gg\gd \mu}u_{j\gd}}{u_{\gg\gg}u_{\gd\gd}}\right) \left(\frac{u_{\gb\nu \mu}}{u_{\gb\gb}u_{\nu\nu}}\right)\\
 						&\ - \frac{u_{i \nu}}{u_{ii}u_{\mu\mu}}\left(\frac{u_{j\nu \mu}}{u_{\nu \nu}} - \frac{u_{j\gd}u_{\gd \nu\mu}}{u_{\nu\nu}u_{\gd\gd}}\right) \left(\frac{u_{i\gb\mu}}{u_{\gb \gb}} - \frac{u_{i\gg}u_{\gg\gb\mu}}{u_{\gb\gb}u_{\gg\gg}}\right).
 					\end{align*}
 			
 					Combining these, simplifying, and comparing to Lemma 	\ref{lem:LW-from-u}, we see that 
 					$$W^{PQ}g^{RS}W_{jP,R}W_{\gb Q,S} + (\Psi_*)_j^Pf_{PQ}(\Psi_*)^Q_\gb = LW_{j\gb}.$$
 			
 				\item[Case 3 $(+/+)$]
 					Now, moving on to the final component where $A = j$ and $B=l$, we expand the right-hand side of (\ref{eq:LW-from-W})
 					\begin{align*}
 						W^{PQ}g^{RS}W_{jP,R}W_{lQ,S} =&\ u^{ii}\left(\frac{W_{ij,k}W_{il,k}}{u_{kk}}  - \frac{W_{ij,\mu}W_{il,\mu}}{u_{\mu\mu}}\right)\\
 						&\ + \left(-u_{\nu\rho} + \frac{u_{\nu p}u_{p\rho}}{u_{pp}}\right)\left(\frac{W_{j\nu,k}W_{l\rho,k}}{u_{kk}} - \frac{W_{j\nu,\mu}W_{l\rho,\mu}}{u_{\mu\mu}}\right)\\
 						&\ + \frac{u_{i \nu}}{u_{ii}}\left(\frac{(W_{ij,k}W_{l\nu,k} + W_{j\nu,k}W_{il,k})}{u_{kk}} - \frac{(W_{ij,\mu}W_{l\nu,\mu} + W_{j\nu,\mu}W_{il,\mu})}{u_{\mu\mu}}\right).
 					\end{align*}
 			
 					The first row is
 					\begin{align*}
 						u^{ii}\left(\frac{W_{ij,k}W_{il,k}}{u_{kk}}  - \frac{W_{ij,\mu}W_{il,\mu}}{u_{\mu\mu}}\right) =&\ \frac{1}{u_{ii}u_{kk}}\left(u_{ijk} - \frac{u_{i\gg k}u_{j\gg}}{u_{\gg\gg}} - \frac{u_{i\gg}u_{j\gg k}}{u_{\gg\gg}} + \frac{u_{i\gg}u_{\gg\gd k}u_{j\gd}}{u_{\gg\gg}u_{\gd\gd}}\right)\\
 						&\ \quad \times \left(u_{ilk} - \frac{u_{i\nu k}u_{l\nu}}{u_{\nu\nu}} - \frac{u_{i\nu}u_{l\nu k}}{u_{\nu\nu}} + \frac{u_{i\nu}u_{\nu\ep k}u_{l\ep}}{u_{\nu\nu}u_{\ep\ep}}\right)\\
 						&\  - \frac{1}{u_{ii}u_{\mu\mu}}\left(u_{ij\mu} - \frac{u_{i\gg \mu}u_{j\gg}}{u_{\gg\gg}} - \frac{u_{i\gg}u_{j\gg \mu}}{u_{\gg\gg}} + \frac{u_{i\gg}u_{\gg\gd \mu}u_{j\gd}}{u_{\gg\gg}u_{\gd\gd}}\right)\\
 						&\ \quad \times \left(u_{il\mu} - \frac{u_{i\nu \mu}u_{l\nu}}{u_{\nu\nu}} - \frac{u_{i\nu}u_{l\nu \mu}}{u_{\nu\nu}} + \frac{u_{i\nu}u_{\nu\ep \mu}u_{l\ep}}{u_{\nu\nu}u_{\ep\ep}}\right).
 					\end{align*}
 					The second row is 
 					\begin{align*}
 						\left(-u_{\nu\rho} + \frac{u_{\nu p}u_{p\rho}}{u_{pp}}\right) &\ \left(\frac{W_{j\nu,k}W_{l\rho,k}}{u_{kk}} - \frac{W_{j\nu,\mu}W_{l\rho,\mu}}{u_{\mu\mu}}\right)\\
 						=&\ \frac{1}{u_{kk}} \left(-u_{\nu\rho} + \frac{u_{\nu p}u_{p\rho}}{u_{pp}}\right)\left(\frac{u_{j\nu k}}{u_{\nu \nu}} - \frac{u_{j\gg}u_{\gg \nu k}}{u_{\nu\nu}u_{\gg\gg}}\right)\left(\frac{u_{l\rho k}}{u_{\rho \rho}} - \frac{u_{l\gd}u_{\gd \rho k}}{u_{\rho\rho}u_{\gd\gd}}\right) \\
 						&\ - \frac{1}{u_{\mu\mu}} \left(-u_{\nu\rho} + \frac{u_{\nu p}u_{p\rho}}{u_{pp}}\right)\left(\frac{u_{j\nu \mu}}{u_{\nu \nu}} - \frac{u_{j\gg}u_{\gg \nu \mu}}{u_{\nu\nu}u_{\gg\gg}}\right)\left(\frac{u_{l\rho \mu}}{u_{\rho \rho}} - \frac{u_{l\gd}u_{\gd \rho \mu}}{u_{\rho\rho}u_{\gd\gd}}\right).
 					\end{align*}
 					The third line is 
 					\begin{align*}
 						\frac{u_{i \nu}}{u_{ii}}&\ \left(\frac{(W_{ij,k}W_{l\nu,k} + W_{j\nu,k}W_{il,k})}{u_{kk}} - \frac{(W_{ij,\mu}W_{l\nu,\mu} + W_{j\nu,\mu}W_{il,\mu})}{u_{\mu\mu}}\right)\\
 						=&\ \frac{u_{i \nu}}{u_{ii}u_{kk}}\left(u_{ijk} - \frac{u_{i\gg k}u_{j\gg}}{u_{\gg\gg}} - \frac{u_{i\gg}u_{j\gg k}}{u_{\gg\gg}} + \frac{u_{i\gg}u_{\gg\gd k}u_{j\gd}}{u_{\gg\gg}u_{\gd\gd}}\right)\left(\frac{u_{l\nu k}}{u_{\nu \nu}} - \frac{u_{l\gd}u_{\gd \nu k}}{u_{\nu\nu}u_{\gd\gd}}\right)\\
 						&\ + \frac{u_{i \nu}}{u_{ii}u_{kk}}\left(u_{ilk} - \frac{u_{i\gg k}u_{l\gg}}{u_{\gg\gg}} - \frac{u_{i\gg}u_{l\gg k}}{u_{\gg\gg}} + \frac{u_{i\gg}u_{\gg\gd k}u_{l\gd}}{u_{\gg\gg}u_{\gd\gd}}\right)\left(\frac{u_{j\nu k}}{u_{\nu \nu}} - \frac{u_{j\gd}u_{\gd \nu k}}{u_{\nu\nu}u_{\gd\gd}}\right)\\
 						&\ -\frac{u_{i \nu}}{u_{ii}u_{\mu\mu}}\left(u_{ij\mu} - \frac{u_{i\gg \mu}u_{j\gg}}{u_{\gg\gg}} - \frac{u_{i\gg}u_{j\gg \mu}}{u_{\gg\gg}} + \frac{u_{i\gg}u_{\gg\gd \mu}u_{j\gd}}{u_{\gg\gg}u_{\gd\gd}}\right)\left(\frac{u_{l\nu \mu}}{u_{\nu \nu}} - \frac{u_{l\gd}u_{\gd \nu \mu}}{u_{\nu\nu}u_{\gd\gd}}\right)\\
 						&\ -\frac{u_{i \nu}}{u_{ii}u_{\mu\mu}}\left(u_{il\mu} - \frac{u_{i\gg \mu}u_{l\gg}}{u_{\gg\gg}} - \frac{u_{i\gg}u_{l\gg \mu}}{u_{\gg\gg}} + \frac{u_{i\gg}u_{\gg\gd \mu}u_{l\gd}}{u_{\gg\gg}u_{\gd\gd}}\right)\left(\frac{u_{j\nu \mu}}{u_{\nu \nu}} - \frac{u_{j\gd}u_{\gd \nu \mu}}{u_{\nu\nu}u_{\gd\gd}}\right).
 					\end{align*}
 			\end{description}
 		
 		Combining these, simplifying, and comparing to Lemma \ref{lem:LW-from-u}, we see that 
 		$$LW_{jl} = W^{PQ}g^{RS}W_{jP,R}W_{lQ,S} + (\Psi_*)_j^Pf_{PQ}(\Psi_*)^Q_l.$$
 		\end{proof}
 	\end{proposition}
 	\begin{remark}
 		This implies, that $L(vWv^T)\geq 0$ for every $v \in \R^{n}$ when $f\equiv const.$, which was shown in \cite{streets_evans-krylov_2016} and is the crucial step which makes the application of the Krylov maximum principle possible.
 	\end{remark}
 	
	\begin{proposition}
		The collection of functions
		$$S_{ABC} \coloneq W_{AB,R}(\Psi_*)^R_C$$
		where $W$ and $\Psi_*$ is given by 
		$$\Psi_* = \begin{pmatrix}
			I & 0\\
			-(D_y^2u)^{-1}(D_{xy}^2u) & (D_y^2u)^{-1}
		\end{pmatrix}$$
		is totally symmetric at any point at which $g$ has been diagonalized.
		\begin{proof}
			Given that we are dividing 2 types of indices among 3 index locations, there are 8 different distinguishable index types. Considering that symmetry moves cannot change type, these 8 distinguishable types divide into 4 cases $(3+,0)$, $(2+, 1-)$, $(1+,2-)$, $(0+,3-)$. The first and fourth reduce to computing one formula and observing. Cases 2 and 3 require computing 3 formulas (one for each of the distinguishable index types) and comparing.
			\begin{description}
				\item[Case 1]$+/+/+$	
				\begin{align*}
					S_{ijk} =&\ W_{ij,R}(\Psi_*)^R_k\\
					 =&\ W_{ij,k} - W_{ij,\ga}u_{k\ga}u^{\ga\ga}\\
					 =&\ \left(u_{ijk} - \frac{u_{ik\gg}u_{j\gg}}{u_{\gg\gg}} - \frac{u_{i\gg}u_{jk\gg}}{u_{\gg\gg}} + \frac{u_{i\gg}u_{k\gg\gd}u_{j\gd}}{u_{\gg\gg}u_{\gd\gd}}\right)\\
					 &\ - u_{k\ga}\left(\frac{u_{ij\ga}}{u_{\ga\ga}} - \frac{u_{i\gg\ga}u_{j\gg}}{u_{\ga\ga}u_{\gg\gg}} - \frac{u_{i\gg}u_{j\gg\ga}}{u_{\ga\ga}u_{\gg\gg}} + \frac{u_{i\gg}u_{\gg\gd\ga}u_{j\gd}}{u_{\ga\ga}u_{\gg\gg}u_{\gd\gd}}\right)
				\end{align*}
				This is symmetric in $i,j,k$, so the result is proved.
							
				\item[Case 2]$+/+/-$
				$$S_{ij\ga} = W_{ij,R}(\Psi_*)^R_\ga = W_{ij,\ga}u^{\ga\ga} = \frac{u_{ij\ga}}{u_{\ga\ga}} - \frac{u_{i\gg\ga}u_{j\gg}}{u_{\ga\ga}u_{\gg\gg}} - \frac{u_{i\gg}u_{j\gg\ga}}{u_{\ga\ga}u_{\gg\gg}} + \frac{u_{i\gg}u_{\gg\gd\ga}u_{j\gd}}{u_{\ga\ga}u_{\gg\gg}u_{\gd\gd}}$$
				By symmetry, $S_{ji\ga}$.
				$$S_{i\ga j} = W_{i\ga, R}(\Psi_*)^R_j = W_{i\ga,j} - W_{i\ga,\gb}u_{j\gb}u^{\gb\gb} = \frac{u_{ij\ga}}{u_{\ga\ga}} - \frac{u_{i\gg}u_{j\gg\ga}}{u_{\ga\ga}u_{\gg\gg}} - \frac{u_{i\ga\gg}u_{j\gg}}{u_{\ga\ga}u_{\gg\gg}} + \frac{u_{i\gg}u_{\ga\gg\gd}u_{j\gd}}{u_{\ga\ga}u_{\gg\gg}u_{\gd\gd}}$$
				By symmetry, this is also $S_{j\ga i}$ and also $S_{\ga j i}$.
				
				\item[Case 3]$+/-/-$
				$$S_{i\ga\gb} = W_{i\ga,R}(\Psi_*)^R_\gb = W_{i\ga,\gb}u^{\gb\gb} = \frac{u_{i\ga\gb}}{u_{\ga\ga}u_{\gb\gb}} - \frac{u_{i\gg}u_{\gg\ga\gb}}{u_{\ga\ga}u_{\gb\gb}u_{\gg\gg}}$$
				By symmetry, it is clear that this must be equal to $S_{\ga i\gb}$ and also $S_{i\gb\ga}$. The only remaining term is then
				$$S_{\ga\gb i} = W_{\ga\gb, R}(\Psi_*)^R_i	= W_{\ga\gb, i} - W_{\ga\gb, \gg}u_{i\gg}u^{\gg\gg} = \frac{u_{\ga\gb i}}{u_{\ga\ga}u_{\gb\gb}} - \frac{u_{\ga\gb\gg}u_{i\gg}}{u_{\ga\ga}u_{\gb\gb}u_{\gg\gg}}$$
				
				\item[Case 4]$-/-/-$
				So we have 
				$$S_{\ga\gb\gg} = W_{\ga\gb,R}(\Psi_*)^R_\gg = W_{\ga\gb,\gg}u^{\gg\gg} = \frac{u_{\ga\gb\gg}}{u_{\ga\ga}u_{\gb\gb}u_{\gg\gg}}.$$
				This is obviously symmetric.
			\end{description}
		\end{proof}
	\end{proposition}
	
	\begin{corollary}\label{cor:LW-from-S}
		Applying the prior two propositions and the fact that $g^{-1} = (\Psi_*) W^{-1} (\Psi_*)^T$ gives 
		$$LW_{AB} = W^{PQ}W^{RS}S_{APR}S_{BQS} + (\Psi_*)_A^Pf_{PQ}(\Psi_*)^Q_B$$
        at any point where $g$ is diagonalized.
	\end{corollary}

\end{document}